\documentclass[12pt, one side,emlines]{amsart}
\usepackage{amssymb,latexsym,xy,eucal,mathrsfs, graphicx, tikz}
\usepackage{amsmath,amssymb,amsfonts,textcomp}
\usepackage{amscd}
\usepackage{amsthm}
\textwidth=20cm \textheight=30cm \theoremstyle{plain}
\newtheorem{lemma}{Lemma}[section]
\newtheorem{theorem}[lemma]{Theorem}
\newtheorem{proposition}[lemma]{Proposition}
\newtheorem{corollary}[lemma]{Corollary}
\usepackage{graphicx}
\theoremstyle{definition}
\newtheorem{example}[lemma]{Example}
\newtheorem{remark}[lemma]{Remark}

\newtheorem{definition}[lemma]{Definition}

\usepackage[bottom=1.0in,top=1.5in]{geometry}
\numberwithin{equation}{section}  \voffset
-55truept \hoffset -15truept
\begin{document}
	\baselineskip 15truept
\title{Generalized zero-divisor graph of $*$-rings  }
\subjclass[2010]{Primary 16W10; Secondary 05C25, 05C15 } 
\author[Anita Lande]
{Anita Lande$^{1}$}	
\address{\rm $1$ Department of Mathematics, Abasaheb Garware College, Pune-411 004, India.}
\email{\emph{anita7783@gmail.com}}

\author[Anil Khairnar]
{Anil Khairnar$^{2}$}	
\address{\rm $2$ Department of Mathematics, Abasaheb Garware College, Pune-411 004, India.}
\email{\emph{anil\_maths2004@yahoo.com, ask.agc@mespune.in}}

\maketitle 
\footnotetext[1]{Corresponding author}
\begin{abstract} Let $R$ be a ring with involution $*$ and $Z^*(R)$ denotes the set of all non-zero zero-divisors of $R$.
  We associate a simple (undirected) graph $\Gamma'(R)$ with vertex set  $Z^*(R)$  and two distinct vertices $x$ and $y$ are adjacent in $\Gamma'(R)$ if and only if $x^ny^*=0$ or $y^nx^*=0$, for some positive integer $n$. We find the diameter and girth of $\Gamma'(R)$. The characterizations are obtained for $*$-rings   having $\Gamma'(R)$ a connected graph, a complete graph, and a star graph. Further, we have shown that for a ring $R$, there is an involution on $R\times R$ such that $\Gamma'(R\times R)$ is disconnected if and only if $R$ is an integral domain.
	
\end{abstract}

 \noindent {\bf Keywords:}  $*$-ring, involution, projections, Artinian ring.
\section{Introduction}

  The concept of the zero-divisor graph for a commutative ring was first introduced by I. Beck \cite{beck1988coloring} in $1988$. He defined the zero-divisor graph of a ring $R$ as a graph on vertex set $R$ and two distinct vertices $x$ and $y$ in $R$ are adjacent if and only if $xy=0$.  He was interested in a ring for which its zero-divisor graph has the same chromatic number and clique number. This is known as Beck's coloring conjecture. Being motivated by Beck, Anderson, and Livingston \cite{anderson1999zero} defined the zero-divisor graph for a commutative ring denoted by $\Gamma(R)$ as a simple (undirected) graph, whose vertex set is a set of  non-zero zero-divisors $Z^*(R)$ of $R$      and two  vertices $x$ and $y$ in $Z^*(R)$ are adjacent if and only if $xy=0$.  Redmond \cite{redmond2002zero} defined  the zero-divisor graph for non-commutative ring $R$  denoted by $\overline{\Gamma(R)}$ to be a simple (undirected) graph, whose vertex set is   $Z^*(R)$ and two distinct vertices $x$ and $y$ in $Z^*(R)$ are adjacent if and only if $xy=0$ or $yx=0$. They studied the diameter, girth, and connectivity for the zero-divisor graphs.  
  
   \par  A $*$-ring is  a ring with a map $*$ called an involution  satisfying  $(x + y)^* = x^*+y^*$, $(xy)^* = y^*x^*$, and $(x^*)^*=x$, for all $x,y\in R$. An element $e$ in a $*$-ring $R$   is called  a projection if $e^2=e=e^*$. An involution $*$ is called a proper involution if $xx^*=0$ then $x=0$. In a $*$-ring $R$ with $S\subseteq R$, $ann_l(S)=\left\{x\in R \colon xs=0~\text{for all}~s\in S\right\}$ $(ann_r(S)=\left\{x\in R \colon sx=0~\text{for all}~s\in S\right\})$ called as set of all left annihilators of $S$ (set of all right annihilators of $S$). 
  \par
     Berberian \cite{berberian1972baer}, defined a Rickart $*$-ring.    A  Rickart $*$-ring  is a $*$-ring such that the right annihilator (the left annihilator) of every element is generated as a right ideal (left ideal), by a projection in $R$. i.e. $\forall x\in R$, $r_R(x)=eR$ ($l_R(x)=Re$) for some projection $e$ in $R$. According to (\cite{cui2018pi}, Definition 2.1), a $*$-ring $R$ is called a generalized Rickart $*$-ring if for any $x\in R$,
      there exists a positive integer $n$ such that $r_R(x^n)=gR$, for some projection $g$ of $R$. Note that $l_R(x^n)=(r_R((x^n)^*))^*$. Thus, generalized Rickart $*$-rings is left-right symmetric. 
     
\par In a graph $G$,  the distance between  two distinct vertices $x$ and $y$ of $G$, denoted by  $d(x,y)$ is the length of the shortest path from $x$ to $y$, otherwise $d(x,y)=\infty$. The diameter of a graph $G$ is defined as $\text{diam}(G)=\sup \{d(x,y)~|~ x ~\text{and}~ y~ \text{are~ vertices ~of}~ G\}$. The girth of $G$, denoted by $\text{gr}(G)$, is the length of the shortest cycle in $G$. Girth of $G$ is $\infty$ if $G$ contains no cycle. We will use the notation $x\leftrightarrow y$ to denote $x$ and $y$ are adjacent. The authors refer to \cite{berberian1972baer}, \cite{atiyah2018introduction}, \cite{godsil2001algebraic}   for the basic definitions and concepts in $*$-rings,  ring theory, and graph theory. In (\cite{samei2007zero},\cite{akbari2004zero}, \cite{akbari2006zero}, \cite{anderson2003zero}) authors studied the  properties of  zero-divisor graphs for commutative rings. In (\cite{redmond2002zero},\cite{redmond2001generalizations}, \cite{patil2018zero}, \cite{khairnar2016zero}, \cite{kumbhar2023strong}) authors studied properties of zero-divisor graph for non-commutative rings.
 In \cite{patil2018zero}, Patil and Waphare introduced the concept of the zero-divisor graph for the rings with involution. 
 Let $R$ be a ring with involution $*$,  $Z(R),~ Z^*(R)$ denotes the set of all zero-divisors and  the set of all non-zero zero-divisors of $R$. For a $*$-ring $R$ the zero-divisor graph $\Gamma^*(R)$ is a graph  with vertex set $Z^*(R)$ and  two distinct elements $x$ and $y$ are adjacent  if and only if  $xy^*=0$. 
 \par  Motivated by the definition of the zero-divisor graph of $*$-ring  in \cite{patil2018zero} and the definition of generalized Rickart $*$-ring in \cite{cui2018pi}, in the second section we introduce the generalized zero-divisor graph denoted by $\Gamma'(R)$ for a $*$-ring $R$. Examples are provided to show that the generalized zero-divisor graph is different from $\Gamma^*(R), \Gamma(R)$, and $\overline{\Gamma(R)}$. We discuss the connectedness, the diameter, and the girth of $\Gamma'(R)$. We express $\Gamma'(R)$ as a generalized join of an induced subgraph of $\Gamma^*(R)$ and the complete graph on non-zero nilpotent elements in $R$.
 In the third section, we have characterized  $*$-rings for which $\Gamma'(R)$ is a complete graph, and characterized $*$-rings for which $\Gamma'(R)$ is a star graph.  We give the class of $*$-rings for which $\Gamma'(R)$ is disconnected.

\section{Diameter and girth of $\Gamma'(R)$ }
In this section, for a $*$-ring $R$, we define the generalized zero-divisor graph  denoted by $\Gamma'(R)$. We discuss the diameter and the girth of the graph $\Gamma'(R)$.

\begin{definition} {\label{def 2}}
Let $R$ be a $*$-ring, the generalized zero-divisor graph denoted by $\Gamma^{'}(R)$ is a simple (undirected) graph with vertex set $Z^*(R)$, and two distinct vertices $x$ and $y$ are adjacent if and only if $x^ny^*=0$ or $y^nx^*=0$ for some positive integer $n$.
\end{definition}
\begin{example}
Let  $\mathbb{Z}_8$ be a ring with identity mapping as an involution.  Note that, $\Gamma^*(\mathbb{Z}_8)=\Gamma(\mathbb Z_8)$. Observe that (Figure \ref{fig 2}), $\Gamma'(\mathbb Z_8)$ is not isomorphic to $\Gamma^*(\mathbb Z_8)$. Also, it is not isomorphic to $\Gamma(\mathbb Z_8)$. 
\end{example}
\begin{figure}[h] 
	\begin{tikzpicture}[scale=0.8]		

		\node [left] at(5.5,-0.2) {$2$}; 
					    \node [left] at(9.5,-0.2) {$4$}; 
						\node [left] at(7.5,2.3) {$6$};  
						\draw[fill=black](5.5,0)circle(.06);  
						\draw[fill=black](8.5,0)circle(.06); 
						\draw[fill=black](7.0,2)circle(.06);  
				
							\draw (5.5,0)--(8.5,0);
							\draw (8.5,0)--(7.0,2);

						\node [left] at(9,-1) {Figure (a) $\Gamma^{*}(\mathbb Z_{8})=\Gamma(\mathbb Z_8)$ };
		\node [left] at(12.5,-0.2) {$2$}; 
			    \node [left] at(16.5,-0.2) {$4$}; 
				\node [left] at(14.5,2.3) {$6$};  
				\draw[fill=black](12.5,0)circle(.06);  
				\draw[fill=black](15.5,0)circle(.06); 
				\draw[fill=black](14.0,2)circle(.06);  
		
					\draw (12.5,0)--(15.5,0);
					\draw (15.5,0)--(14.0,2);
					\draw (12.5,0)--(14.0,2);

				\node [left] at(16,-1) {Figure (b) $  \Gamma^{'}(\mathbb Z_{8})$ };

	\end{tikzpicture}
\caption{Zero-divisor graphs $\Gamma^*(\mathbb Z_8)$ and $\Gamma^{'}(\mathbb Z_8)$ }
\label{fig 2}
\end{figure}  
\begin{remark}
For a $*$-ring $R$, $\Gamma^*(R)$ is a subgraph of $\Gamma'(R)$.
\end{remark}
\begin{theorem}
Let $R$ be a reduced $*$-ring then $\Gamma^*(R)$ is isomorphic to $\Gamma'(R)$. 
\end{theorem}
\begin{proof}
Let $R$ be a reduced $*$-ring and $x \in R$. Then $r_R(x) = r_R(x^n)$ for any positive integer $n$. Hence $xy^*=0$ if and only if $x^ny^*=0$ for any positive integer $n$. Thus $\Gamma^*(R)$ is isomorphic to $\Gamma'(R)$. 
\end{proof}
We give an example to show that the  generalized zero-divisor graph $\Gamma'(R)$ is different from $\overline{\Gamma(R)}$ defined  in \cite{redmond2002zero}.
\begin{example}
 Let $R=M_2(\mathbb Z_2)$ be a $*$-ring with the transpose of a matrix as an involution. The zero-divisor graph $\overline{\Gamma(R)}$ is depicted in Figure \ref{fig3}, and $\Gamma'(R)$ is depicted in  Figure \ref{fig31}. Observe that $\overline{\Gamma(R)}$ and $\Gamma'(R)$ are not isomorphic.
\end{example}
\begin{figure}[!hbt]
\begin{minipage}[c]{0.5\linewidth}
\centering
\includegraphics[width=4.0cm]{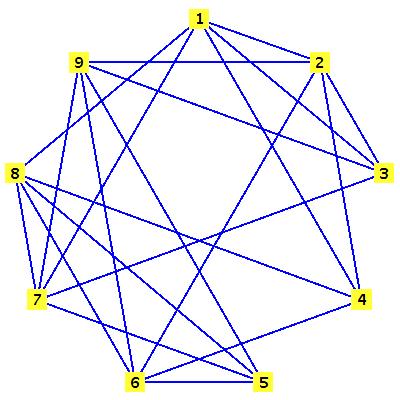}
\caption{$\overline{\Gamma(M_2(\mathbb Z_2))}$}
\label{fig3}
\end{minipage}\hfill
\begin{minipage}[c]{0.5\linewidth}
\centering
\includegraphics[width=4.0cm]{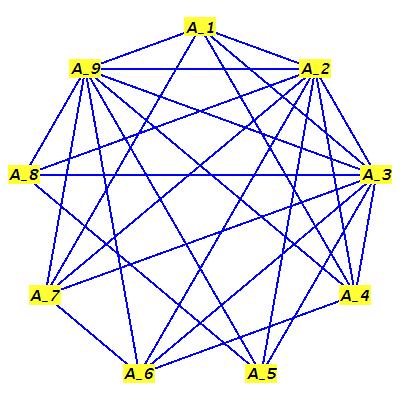}
\caption{$\Gamma'(M_2(\mathbb Z_2))$}
\label{fig31}
\end{minipage}
\end{figure}
Let $N(x)$ denote the neighborhood of $x$ in  $\Gamma'(R)$. In the following lemma, we give  relationship between the neighborhood of $x$ and the neighborhood of $x^k$ for any positive integer $k$. 
\begin{lemma}
Let $R$ be a $*$-ring.  Then
\begin{enumerate} 
\item for any $x\in Z^*(R)$ and positive integer $k$, $N(x)\subseteq N(x^k)$.
\item  For any $x\in Z^*(R)$ and positive integer $k$, $N(x^k)\subseteq N(x^{k+1})$.
\item If $x$ is a non-zero  nilpotent element then  $N(x)= N(x^k)=Z^*(R)$ for any positive integer $k$.
\item  If $x$ is a potent element (i.e., $x^m=x~\text{for some positive integer m})$ then $N(x)=N(x^k)$ for any positive integer $k$.
\end{enumerate}
\end{lemma}
\begin{proof}\noindent
\begin{enumerate}
\item Let  $y$ be adjacent to $x$ in  $\Gamma'(R)$. Then $x^ny^*=0$ or $y^nx^*$ for some positive integer $n$. Therefore $(x^k)^ny^*=0$ or $y^n(x^k)^*=0$. Hence 
$y$ is adjacent to $x^k$. Thus $N(x)\subseteq N(x^k)$ for any positive integer $k$.
\item Let $y\in N(x^k)$. Then there exists positive integer $i$ such that $(x^{k+1})^iy^*=x^{i}(x^k)^iy^*=0$ or $0=y^i(x^k)^*=y^i(x^{k+1})^*$. Therefore $y\in N(x^{k+1})$. Hence $N(x^k)\subseteq N(x^{k+1})$.
\item If $x$ is a non-zero nilpotent element then there exists positive integer $s$ such that $x^s=0$. Thus $x^sy^*=0$ and $(x^k)^sy^*=0$. Therefore all vertices $y$ are adjacent to $x^k$, for any positive integer $k$. Hence $N(x)= N(x^k)=Z^*(R)$ for any positive integer $k$.
\item Let $k$ be any positive integer and $x^m=x$ for some positive integer $m$.  
If $k\leq m$, then by (2), $N(x^k)\subseteq N(x^m)=N(x)\subseteq N(x^k)$. Thus $N(x)=N(x^k)$.\\
Suppose $m<k$. Note that $x^{2m}=x^mx^m=x^2$.  Therefore $N(x)\subseteq N(x^{2m})= N(x^2)\subseteq N(x^m)=N(x)$. Hence $N(x^{2m})=N(x)$. Inductively, we can show that $N(x^k)=N(x)$ for any positive integer $k$.
\end{enumerate}
\end{proof}
\begin{lemma}{\label{lem5.2}}
Let $R$ be a $*$-ring with unity and $V(\Gamma'(R))=Z^*(R)$ be a vertex set of  $\Gamma'(R)$. Then, 
\begin{enumerate}
\item $x\in Z^*(R)$ if and only if $x^*\in Z^*(R)$.
\item If $x,y$ are adjacent in $\Gamma'(R)$, then $x^i,~ y^j$ are adjacent in $\Gamma'(R)$, where $i,~ j$ are integers.
\item Let $e\in Z^*(R)$ be a projection.  Then $e$ and $x$ are adjacent in $\Gamma'(R)$ if and only if $x^n\in R(1-e)$ for some $n\in \mathbb{N}$.
\item If $e,~f\in Z^*(R)$ are projections, then $e$ and $f$ are adjacent in $\Gamma'(R)$ if and only if $ef=0$ and $fe=0$.
\end{enumerate}
\end{lemma}
\begin{proof}\noindent
\begin{enumerate}
\item If $x\in Z^*(R)$, then there exists $y\in Z^*(R)$ such that $xy=0$. Therefore $y^*x^*=0$. Since $y\neq 0$,  $y^*\neq 0$. Thus $x^*\in Z^*(R).$  Similarly if $x^*\in Z^*(R)$ then $x=(x^*)^*\in Z^*(R)$. 
\item Let $x$ and $y$ be adjacent in $\Gamma'(R)$ and $i,~j \in \mathbb N$. Then there exists $n\in \mathbb{N}$ such that $x^ny^*=0$ or $y^nx^*=0$. Therefore $(x^i)^n(y^j)^*=0$ or $(y^j)^n(x^i)^*=0$. Thus $x^i$ and  $y^j$ are adjacent.
\item If $e$ and $x$ are adjacent in $\Gamma'(R)$ then there exists a  positive integer $n$ such that $e^nx^*=ex^*=0$ or $x^ne^*=0$. Therefore $x^ne=0$.  Hence $x^n=x^n(1-e)\in R(1-e)$.
Conversely, if $x^n\in R(1-e)$ for some positive integer $n$, then $x^n=x^n(1-e)$. Therefore $x^ne^*=x^ne=0$. Hence $x$ and  $e$ are adjacent.
\item If $e$ and $f$ are adjacent, then by definition $ef=0$ or $fe=0$. Since $e,f$ are projections,  $ef=0$ if and only if  $fe=f^*e^*=(ef)^*=0$.
\end{enumerate}
\end{proof}

 In $\Gamma^*(R)$ nilpotent elements in $R$ need not be adjacent to all other vertices, see Example \ref{ex 1}. 

 \begin{theorem}\label{thm2.4}
 Let $R$ be a $*$-ring, and $x$ be  a non-zero nilpotent element in  $R$  then $x$ is adjacent to all other vertices in $\Gamma'(R)$.
 \end{theorem}
 \begin{proof}
  Let $x$ be a non-zero nilpotent element in $R$. Then there exists positive integer $n$ such that $x^n=0$. Therefore $x^ny^*=0$ for all $y\in Z^*(R)$. Hence $x$ is adjacent to all $y$ in $\Gamma'(R)$. 
 \end{proof}
\begin{corollary}\label{cor 2.8}
Let $R$ be a ring with involution $*$, if $R$ contains a non-zero nilpotent element then $\Gamma'(R)$ is connected. 
\end{corollary}
\begin{proof}
Follows from  Theorem \ref{thm2.4}.
\end{proof}
For any ring $R$, $\Gamma(R)$ is always connected (\cite{anderson1999zero}, Theorem 2.3).
We give an example of a $*$-ring such that $\Gamma'(R)$ is disconnected.
 \begin{example}
 Let $R=\mathbb Z_3\times \mathbb Z_3$ with the involution $(x,y)^*=(y,x)$. Here, $\Gamma'(R)$ is disconnected as depicted in  Figure \ref{fig 9}.
 \end{example}
 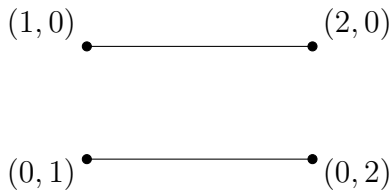
\begin{figure}[h] 
 	\begin{tikzpicture}[scale=1.0]		
 
 		\node [left] at(5.5,-0.2) {$(0,1)$}; 
 					    \node [left] at(9.7,-0.2) {$(0,2)$}; 
 					    \node [left] at(5.5,1.8) {$(1,0)$};
 					     \node [left] at(9.7,1.8) {$(2,0)$};
 					  
 						\draw[fill=black](5.5,0)circle(.06);  
 						\draw[fill=black](8.5,0)circle(.06); 
 						\draw[fill=black](5.5,1.5)circle(.06);  
 					 	\draw[fill=black](8.5,1.5)circle(.06); 
 				
 							\draw (5.5,0)--(8.5,0);
\draw (5.5,1.5)--(8.5,1.5);
 
 					\end{tikzpicture}
 \caption{ $\Gamma'(\mathbb Z_3 \times \mathbb Z_3)$  }
 \label{fig 9}
 \end{figure}  
 In the next theorem, we discuss connectivity and  diameter of $\Gamma'(R)$, where $R$ is a finite $*$-ring  with unity.

\begin{theorem} \label{thm 1}
Let $R$ be a finite $*$-ring with unity and an identity involution. Then $\Gamma'(R)$ is connected and its diameter is 2.
\end{theorem}
\begin{proof}
Let $R$ be a $*$-ring. If $R$ has a non-zero nilpotent element $a$ then it is adjacent to all other vertices in $\Gamma'(R)$. Hence there is a path $x\leftrightarrow a\leftrightarrow y$ between any two distinct vertices $x$ and $y$ of $R$.
If $R$ does not contain any non-zero nilpotent element then $R$ is a reduced ring and hence $\displaystyle R=\prod_{i=1}^{n}D_i$, where $D_i$ are an integral domains. For any two vertices $x=(x_1, x_2,\ldots ,x_n)$ and $ y=(y_1, y_2,\ldots ,y_n)$  which are not adjacent, there exists $i$ and $j$ such that $x_i=0$ and $y_j=0$. Choose $z_i\neq 0$ and $z_j\neq 0$ and $z=(0,0,\ldots, z_i,\ldots, z_j,\ldots, 0)$. Hence there is a path $x\leftrightarrow z\leftrightarrow y$. Thus $\Gamma'(R)$ is connected and the diameter is $2$.
\end{proof}
\begin{theorem}\label{thm 2.10}
Let $R=M_n(\mathbb{Z}_m)$ be a $*$-ring with the transpose of a matrix as an involution then $\Gamma'(R)$ is connected.
Moreover, $\text{diam}(\Gamma'(R))=2$ and $\text{gr}(\Gamma'(R))=3$.
\end{theorem}
\begin{proof}
Let $E_{ij}$ be $n\times n$ matrix with $(i,j)^{th}$ entry $1$ and $0$ otherwise. 
 Clearly, $E_{12}$ is a nilpotent element. Let $X, Y$ be any two vertices in $V(\Gamma'(R))$. If either $X$ or $Y$ is nilpotent then $X$ and $Y$ are adjacent. If none of $X$ or $Y$ is nilpotent then $X\leftrightarrow  E_{12} \leftrightarrow Y$ is a path. Hence $\Gamma'(R)$ is connected and  $\text{diam}(\Gamma'(R))=2$. Next, for any $X\in V(\Gamma'(R))\setminus \{E_{12},E_{21}\}$, $\{E_{12}, E_{21}, X\}$ forms a clique in $\Gamma'(R)$. Hence $\text{gr}(\Gamma'(R))=3$.
\end{proof}

We provide an example of a $*$-ring $R$ for which $\Gamma^*(R)$ is disconnected but $\Gamma'(R)$ is connected. 
\begin{example}\label{ex 1}
Let $R=M_2(\mathbb Z_2)$ be a $*$-ring with the transpose of a matrix as an involution. By Corollary \ref{cor 2.8}, $\Gamma'(R)$ is connected. Let $E_{ij}$ denote $2\times 2$ matrix with $(i,j)^{th}$ entry $1$ and $0$ otherwise. The set of the non-zero zero-divisors of $R$ is $Z^*(R)=\{A_1,A_2,A_3,A_4,A_5,A_6,A_7,A_8,A_9\}$, where 
$A_1=E_{11},~ A_2=E_{12},~A_3=E_{21},~A_4=E_{22},~ A_5=E_{11}+E_{12},~ A_6=E_{11}+E_{21},~ A_7=E_{12}+E_{22},~ A_8=E_{21}+E_{22},~ A_9=A_6+A_7$.
As depicted in  Figure \ref{M1}, $\Gamma^*(R)$ is disconnected. 
\end{example}
\begin{figure}[!hbt]
\includegraphics[width=5cm]{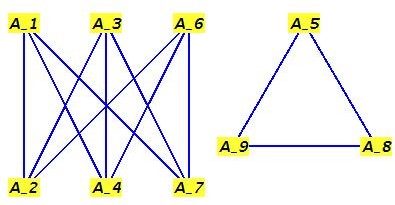}
\caption{$\Gamma^*(M_2(\mathbb Z_2))$}
\label{M1}
\end{figure}
In the following proposition, we find the girth of $\Gamma'(R)$.
\begin{proposition}\label{prop 1}
Let $R$ be a finite $*$-ring. Then $\text{gr}(\Gamma'(R))=3~\text{or}~4~\text{or}~\infty$.
\end{proposition}
\begin{proof}
If $R$ is a reduced ring, then $R=\prod_{i=1}^{k}D_i$, and each $D_i$ is an integral domain. 
If  $R=D_1\times D_2$  then $\Gamma'(R)=K_{m,n}$ which is a complete bipartite graph, where $|D_1|=m, ~|D_2|=n$. It does not have an odd cycle. Hence its girth is greater than or equal to 4. If $m\geq 2$ and $n\geq 2$ then the girth is exactly 4. If $m=1$ or $n=1$ then girth is $\infty$.
If $R=D_1\times D_2\times D_3$  then the girth of  $\Gamma'(R)$ is $3$.\\
If $R$ is not a reduced ring then $R$ has a nilpotent element say $a$. If  $a^3\neq 0$ then elements $a,~a^2,~a^3$ form a clique in  $\Gamma'(R)$. If $a^3=0$ and $a^2\neq 0$,  then elements $1-a^2,~ 1+a^2, ~a$ are distinct and form  a clique in  $\Gamma'(R)$. If $a^2=0$, then $1-a,~ 1+a,~ a$ are distinct and form  a clique in  $\Gamma'(R)$. Hence  girth of $\Gamma'(R)$ is 3.
\end{proof}
We provide an example of a $*$-ring $R$ for which the diameter of $\Gamma^*(R)$ and $\Gamma'(R)$ are not equal.
\begin{example} 
Let $R=M_2(\mathbb Z_{21})$ with the transpose of a matrix as an involution. By [\cite{patil2018zero}, Example 3.1], $\text{diam}(\Gamma^*(R))=4$. But by  Theorem \ref{thm 2.10}, $\text{diam}(\Gamma'(R))=2$.
\end{example}

To express $\Gamma'(R)$ as a generalized join of graphs recall the  following definition from \cite{cardoso2013spectra}.

\begin{definition} Let $G_1, G_2,\ldots, G_n$ be graphs and $H$ be a graph  with $n$ vertices $ \left\{1,2,\ldots,n\right\}$. The $H$-generalized join of the graphs $G_1, G_2,\ldots,G_n$ is denoted by $\displaystyle\bigvee_{H}\{G_1, G_2,\ldots,G_n\}$ obtained by replacing each vertex $i$ of $H$ by the graph $G_i$ and joining any two vertices of $G_i$ and $G_j$ if and only if  vertices $i$ and $j$ are adjacent in $H$. \end{definition}
Let $R=M_2(F)$ be a matrix ring over a finite field F. In the following theorem, $\Gamma'(R)$ is expressed as a generalized join of  the induced subgraph of  $\Gamma^*(R)$ and the complete graph with  non-zero nilpotent elements in $R$ as vertices.
\begin{theorem}
Let $F$ be a finite field and $R=M_2(F)$. Let $N$ be the set of all non-zero nilpotent elements in a ring $R$. Then $$ \Gamma'(R)=\Gamma_1 \bigvee_{K_2} K_{|N|},$$
where $\Gamma_1$ is the induced subgraph of $\Gamma^*(R)$ on the set of all non-nilpotent zero-divisors in $R$.
\end{theorem}
\begin{proof}
Let $A=\begin{bmatrix}a&b\\c&d
\end{bmatrix} \in M_2(F)$ be a zero-divisor then $\det(A)=0$. Which gives $c=a\alpha, d=b\alpha$, for some $\alpha$ and hence $A=\begin{bmatrix} a&b\\ a\alpha & b\alpha
\end{bmatrix}$. 
Then by Caley-Hamilton theorem, $A^2-(a+b\alpha)A=0$.
  Therefore $A^2=(a+b\alpha)A$. Hence $A^k=(a+b\alpha)^{k-1}A$ for any positive integer $k$. If $a+b\alpha =0$ then $A^2=0$, otherwise $A$ is not a nilpotent element.
 If $A$ is a nilpotent element then all vertices in $\Gamma'(R)$ are adjacent to $A$. Suppose $A$ is not a nilpotent element and  $B$ is any non-nilpotent element, then $A$ and $B$ are adjacent if and only if there exists positive integer $n$ such that $A^nB^*=(a+b\alpha)^{n-1}AB^*=0$ or $B^nA^*=(c+d\beta)^{n-1}BA^*=0$. Thus $A$ and $B$ are adjacent in $\Gamma'(R)$ if and only if $AB^*=0$. Hence $A$ and $B$ are adjacent in $\Gamma'(R)$ if and only if they are adjacent in $\Gamma_1$. A nilpotent element in $R$ is adjacent to all other elements in $\Gamma'(R)$. Thus $\Gamma'(R)$ is a $K_2$-join of $\Gamma_1$ and $K_{|N|}$. 
 \end{proof}
 As an illustration of the above  theorem, we express $\Gamma'(M_2(\mathbb Z_2))$ as a generalized join of two graphs. 
 \begin{example}
 Let $R=M_2(\mathbb Z_2)$ be a $*$-ring with the transpose of a matrix as an involution. The graphs $\Gamma^*(R)$ and $\Gamma'(R)$ are depicted in  Figure \ref{M1} and Figure  \ref{fig M3} respectively. Let $\Gamma_1$ be the induced subgraph of $\Gamma^*(R)$ on the set of all non-nilpotent zero-divisors in $R$, depicted in Figure \ref{fig M1}. Let $N$ be the set of all non-zero nilpotent elements in $R$, and $K_{|N|}$ be the complete graph on $|N|$ vertices, depicted in Figure \ref{fig M2}.
 Then $\Gamma'(R)$ is $K_2$ generalized join of two graphs $\Gamma_1$  and $K_{|N|}$ as shown in Figure \ref{fig M3}.
 \end{example}
 \begin{figure}[!hbt]
 \begin{minipage}[c]{0.3\linewidth}
 \centering
 \includegraphics[width=2.0cm]{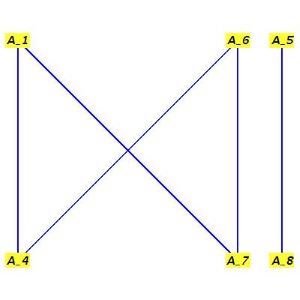}
 \caption{$\Gamma_1$}
 \label{fig M1}
 \end{minipage}\hfill
 \begin{minipage}[c]{0.3\linewidth}
 \centering
 \includegraphics[width=2.0cm]{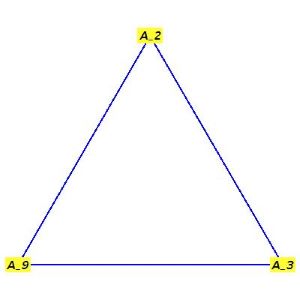}
 \caption{$K_{|N|}$}
 \label{fig M2}
 \end{minipage}
  \begin{minipage}[c]{0.3\linewidth}
  \centering
  \includegraphics[width=3.0cm]{M_2_Z_2_1}
  \caption{$\Gamma'(R)$}
  \label{fig M3}
  \end{minipage}
 \end{figure}

In the following result, we give a condition on $\Gamma'(R)$ so that the ring $R$ is finite.
\begin{lemma}
If $R$ is a $*$-ring such that every vertex of $\Gamma^{'}(R)$ has a finite degree, then $R$ is a finite ring.
\end{lemma}
\begin{proof}
Let $R$ be a $*$-ring and every vertex of $\Gamma'(R)$ has a finite degree. Suppose  $R$ is  infinite. Let $x$ and $y$ denote non-zero elements of $R$ such that $x^ny^*=0$ or $y^nx^*=0$ for some positive integer $n$. Then $y^*R\subseteq r_R({x^n})$. If $y^*R$ is infinite, then $x^n$ has an infinite degree in $\Gamma^{'}(R)$, a contradiction. Hence $y^*R$ is finite. Therefore there exists an infinite subset $B=\left\{a_0,a_1,a_2,\dots \right\}$ of $R$ such that  $y^*a_0=y^*a_i$ for all $i\in \mathbb{N}$. Hence $\{a_0-a_i~|~i\in \mathbb{N} \}$ is an infinite subset of $r_R({y^*})$ and so $y^*$ has an infinite degree in $\Gamma^{'}(R)$, a contradiction. Thus, $R$ must be finite.
\end{proof}

\section{Connectedness and Completeness of $\Gamma'(R)$ }
In this section,  we give a characterizations for completeness and  connectedness of $\Gamma'(R)$. 
In the following lemma, we identify nilpotent elements and projections in $R$ from $\Gamma'(R)$.
\begin{lemma}\label{lem5.3}
Let $R$ be a $*$-ring with unity and $a\in Z^*(R)$ be an element that is adjacent to all other elements of $Z^*(R)$, then   $a$ is a nilpotent element or a projection. Further if $a$ is a projection then it is  orthogonal to all other non-trivial projections.
\end{lemma}
\begin{proof}
Suppose $a \in Z^*(R)$ and $a$ and  $x$ are adjacent for any $x\in Z^*(R)$.  If $a$ is a nilpotent element then we are done. Assume $a$ is not a nilpotent element. We will show $a=a^*=a^2$. If $a\neq a^*$ then $a^2=a(a^*)^*=0$. Therefore $a$ is nilpotent,  a contradiction. Hence $a=a^*$. Now, we will show $a^2=a$. If  $a^2\neq a$ and $a^2\neq 0$ then $a^3=a(a^2)^*=aa^2=0$. Therefore $a$ is nilpotent, a contradiction. Hence $a^2=a$. Thus $a$ is a projection.\\
Let $e$ be a projection that is adjacent to all vertices  in $\Gamma'(R)$.  Note that for any other nontrivial projection $f$ in $R$, $ef=ef^*=0=fe$, because $e$ is adjacent to $f$. Thus $e$ is orthogonal to $f$.
\end{proof}

An ideal $I$ of a $*$-ring $R$ is a $*$-ideal if $a^*\in I$ whenever $a\in I$ (see \cite{berberian1972baer}).

The following theorem gives a characterization for  $\Gamma^{'}(R)$ to be  a complete graph. 
\begin{theorem}
Let $R$ be an abelian  $*$-ring with unity. Then $\Gamma^{'}(R)$ is complete if and only if either $R\simeq \mathbb{Z}_2 \oplus \mathbb{Z}_2 $ or  $Z^*(R)$ is a set of  nilpotent elements in $R$. Further, if $R$ is a commutative ring  and $\Gamma'(R)=K_n$ with $n>2$ then  $Z(R)$ is a $*$-ideal of $R$.
\end{theorem}
\begin{proof}Let $R$ be an abelian $*$-ring with unity. 
Suppose $\Gamma^{'}(R)$ is complete. Let $a\in Z^*(R)$, then by  Lemma \ref{lem5.3}, $a$ is a central projection or nilpotent. Suppose $a$ is a central projection, then  $R=aR\oplus(1-a)R$. For any $x\neq 0$ in $aR$ we have $x(1-a)=0$. Hence $x\in Z^*(R)$. If $x\neq a$ then   $x=xa=0$, a contradiction. Therefore $aR=\left\{0, a\right\}\simeq \mathbb{Z}_2$. Similarly $(1-e)R\simeq \mathbb{Z}_2$. Hence $R\simeq \mathbb{Z}_2\oplus \mathbb{Z}_2.$ \\
If
 $Z^*(R)$ does not contain any projection, then $Z(R)$ is a set of all  nilpotent elements in $R$. \\
 Conversely, if $R\simeq \mathbb{Z}_2\times \mathbb{Z}_2$ then $\Gamma'(R)=K_2$. Also if $Z(R)$ is the set of all nilpotent elements in $R$  then by  Theorem \ref{thm2.4}, $\Gamma'(R)=K_n$, where $n=|Z(R)|-1$.
 Assume that  $R$ is a commutative ring  and $\Gamma'(R)=K_n$ with $n>2$.  $Z(R)=\sqrt{0}$ is a nilradical of $R$ and hence ideal of $R$. By  Lemma \ref{lem5.2}, it is closed under involution $*$. Hence $Z(R)$ is a $*$-ideal of $R$.
\end{proof}

Next, we give a necessary and sufficient condition on a $*$-ring $R$ such that $\Gamma'(R)$ is a star graph. In connection with that, we need the following lemma. 

\begin{lemma}{\label{lem 2.3}}
Let $R$ be a $*$-ring with unity such that $\Gamma^{'}(R)$ is a star graph with center $e$. If $e^2=0$, then $e^*=e$ and $R$ does not contain a non-trivial idempotent element.
\end{lemma}
\begin{proof}
Let $\Gamma^{'}(R)$ be a star graph with center $e$. Assume that  $e^2=0$.
Since $e$ is a center of a star graph $\Gamma^{'}(R)$, it is adjacent to all vertices in $Z(R)\setminus\{e, 0\}$. 
Note that $e\neq 1+e^*$. If $e=1+e^*$ then $0=e^2=(1+e^*)^2$ this implies $1+2e^*+(e^*)^2=0$, which gives  $2e^*=-1$. Hence $2e=-1$.  Therefore $e=(-1)(-e)=-2ee=-2e^2=0$ a contradiction. Thus we get $e-e^*\neq 1$.\\ 
We will show that $e=e^*$. On the contrary, assume that  $e\neq e^*$. 
Since $e$ is adjacent to all vertices in $\Gamma'(R)$, for any $a\in Z(R)\setminus\{e,0\}$, there exists a positive integer $n$ such that $a^ne^*=0$ or $ea^*=0$. That is there is a positive integer $n$ such that $a^ne^*=0$.  Then $a^n(e^*e)^*=a^ne^*e=0$. Hence, $a$ is adjacent to $e^*e$. But $e$ is the vertex in $\Gamma'(R)$ which is adjacent to all other vertices. Therefore either $e^*e=0$ or $e^*e=e$. If $e^*e=0$ then $e=0$, otherwise $e$ is adjacent to $e$, which is a contradiction. Therefore $e^*e=e$. This gives $ a^ne^*e=a^ne=0$. Therefore $a^ne=0$, that is $a^n(e^*)^*=0$,  which gives  $e^*$ is adjacent to $a$. Thus $e^*$ is the center of $\Gamma'(R)$ and hence $e^*=e$. \\
Now, $x$ be any non-trivial idempotent element in $R$. Then $x(1-x^*)^*=x(1-x)=x-x^2=0$. \\
We prove that $x=1-x^*$. If $x\neq 1-x^*$ then $x$ and $1-x^*$ are adjacent. Since $\Gamma^{'}(R)$ is a star graph with center $e$, we have either $x=e$ or $1-x^*=e$. If $x=e$, we get $0=e^2=x^2=x$. If $1-x^*=e$ then $(1-x^*)^*=e^*=e$ that is $1-x=e$. Then $0=e^2=(1-x)^2=1-2x+x^2=1-2x+x=1-x $ gives $x=1$. Hence in either case, we have $x=0$ or $x=1$, a contradiction to $x\notin \{0,1\}$. Thus $x=1-x^*$.\\
Now, if $x=1-x^*$ then $xx^*=(1-x^*)x^*=x^*-{x^*}^2=(x-x^2)^*=0$.
Since $e$ and $x$ are adjacent, $x^ne^*=0$ for some positive integer $n$. Then $x^n(x+e)^*=x^nx^*+x^ne^*=x^{n-1}xx^*+x^ne=0$. Hence in either case $x+e$ and $x$ are adjacent.
Therefore the vertices $e, x, x+e$ are distinct and form a triangle in $\Gamma^{'}(R)$ a contradiction to the fact  $\Gamma^{'}(R)$ is a star graph. Thus $R$ does not contain a non-trivial idempotent.
\end{proof} 

In the following theorem, we give a characterization for $\Gamma'(R)$ to be a  star graph. 
\begin{theorem}
Let $R$ be an abelian, Artinian $*$-ring with unity such that the involution  $*$  is proper and $\Gamma^{'}(R)$ contains at least three vertices. Then $\Gamma^{'}(R)$ is a star graph if and only if $R=\mathbb{Z}_2\oplus D$, where $D$ is an integral domain with involution $*$. 
\end{theorem}
\begin{proof}
Suppose $\Gamma^{'}(R)$ is a star graph with center $e$. By  Lemma \ref{lem5.3}, $e$ is a central projection or $e$ is a nilpotent element. Suppose $e$ is a central projection, then $R=eR\oplus(1-e)R$. If $a\in eR$ is a non-zero element, then it is adjacent to $1-e$. But $\Gamma^{'}(R)$ is a star graph with center $e$, hence the vertex $1-e$ is adjacent to $e$ only. Therefore $a=e$ and hence $eR=\mathbb{Z}_2$. This gives $R=\mathbb{Z}_2 \oplus (1-e)R$.\\
Let $b\in (1-e)R$ and $xb=0$ for some non-zero $x\in (1-e)R$. Then $x,b^*\neq e$. If $x \neq b^*$, then $x(b^*)^*=xb=0$. Therefore $x$ and $b^*$ are adjacent, a contradiction to $\Gamma^{'}(R)$ is a star graph. Therefore $x=b^*$ i.e. $b=x^*$. This gives $xb=0$ which is $xx^*=0$, hence we get  $x=0$, because $*$ is a proper involution. Hence $(1-e)R$ does not contain non-zero zero-divisors. Thus $(1-e)R$ is an integral domain.\\
Now, we prove that the case $e^2=0$ does not arise. Suppose $e^2=0$. Then by  Lemma \ref{lem 2.3}, $R$ does not contain any non-trivial idempotent and $e^*=e$. First, we prove that every vertex of $\Gamma^{'}(R)$ is a nilpotent element. Let $a\in V(\Gamma^{'}(R))\backslash \{e\}$. On the contrary, assume that $a$ is not a nilpotent element.
 Then $a^n=ra^m$ (as $R$ is  Artinian), for some $r \in R$ and some positive integers $m>n$. This gives $(1-ra^{m-n})a^n=0$ and $1-ra^{m-n}\neq 0$ (otherwise $a$ is a unit). Therefore $a^n$ and $(1-ra^{m-n})$ are vertices in $\Gamma'(R)$ and hence adjacent to $e$. If $a^n=e$ then $a^{2n}=e^2=0$ and $a$ is nilpotent element,  a contradiction. If $(1-ra^{m-n})=e$ or $1-ra^{m-n}=a^n$ then $1=2ra^{m-n}+ra^{m-n}ra^{m-n}$ or $1=ra^{m-n}+a^n$ and hence $a$ is a unit, a contradiction. Therefore $\left\{1-ra^{m-n}, (a^n)^*, e\right\}$ forms a triangle in $\Gamma'(R)$, this contradicts  the fact that $\Gamma'(R)$ is a star graph. Thus $a$ must be nilpotent element. 
 Therefore $Z^*(R)$ is a set of all nilpotent elements in $\Gamma'(R)$. Hence $\Gamma'(R)$ is a complete graph. Since $\Gamma'(R)$ is a star graph, $\Gamma'(R)=K_2$ has only two vertices, a contradiction. Thus $e^2=0$ is not possible.\\
 Conversely, if $R=\mathbb{Z}_2\times D$ and $D$ is the domain then clearly $\Gamma'(R)$ is a star graph with  center $(1,0)$.
\end{proof}
For a ring $R$, the following proposition gives characterization for the  connectedness of $\Gamma'(R\times R)$. 
\begin{proposition}{\label{prop5.8}}
Let $R$ be a commutative ring and $A=R\times R$ with $(x,y)^*=(y,x)$ as an involution. Then $\Gamma^{'}(A)$ is connected if and only if $R$ is not an integral domain.
\end{proposition}
\begin{proof}
Let $R$ be a commutative ring and $A=R\times R$. We prove that $\Gamma^{'}(A)$ is disconnected if and only if $R$ is an integral domain.\\
Suppose $R$ is an integral domain. Let $x, y\in R\setminus 0$. Therefore for any positive integer $n$, $(x,0)^n(0,y)^*=(x^n,0)(y,0)=(x^ny,0)\neq (0,0)$ and  $(0,y)^n(x,0)^*=(0,y^n)(0,x)=(0,y^nx)\neq (0,0)$. Also, $(x,0)^n(y,0)^*=(x^n,0)(0,y)=(0,0)$ for any positive integer $n$. Therefore  $(R\setminus\{0\})\times \{0\}$ and $\{0\}\times (R\setminus\{0\})$ are components of the graph $\Gamma'(A)$.   Hence $\Gamma^{'}(A)$ is disconnected.\\
Conversely, suppose that $\Gamma^{'}(A)$ is disconnected. If $R$ is not an integral domain. Then there exists non-zero elements $x, y$ in $R$ such that $xy=0$. Let $(a,b)$ and $(c,d)$ be any two vertices in $\Gamma'(A)$. Therefore there exists $(u,v)$ and $(s,t)$ in $R\times R$ such that $(a,b) \leftrightarrow (u,v)$ and $(s,t) \leftrightarrow (c,d)$ are edges. Now $u\neq 0$ or $v\neq 0$ and $s\neq 0$ or $t\neq 0$.
If $u\neq 0$ and $t\neq 0$ then $(a,b) \leftrightarrow (u,0) \leftrightarrow (x,0) \leftrightarrow (0,y) \leftrightarrow (0,t) \leftrightarrow (c,d)$  is a path. If $u\neq 0$ and $s\neq 0$ then $(a,b) \leftrightarrow (u,0) \leftrightarrow (s,0) \leftrightarrow (c,d)$ is a path. Similarly, in other cases, we can find a path between $(a,b)$ and $(c,d)$. Hence $\Gamma'(A)$ is connected, a contradiction. Hence $R$ must be an integral domain. 
\end{proof}

From  Proposition \ref{prop5.8}, it is clear that $\Gamma^{'}(R)$ need not be connected in general. If a ring $R$ is decomposable as the direct sum of two  rings  with component-wise involution then $\Gamma^{'}(R)$ is connected.
\begin{theorem}{\label{thm5.9}}$[$\cite{patil2018zero}, Theorem 2.4$]$
Let $R_1,~ R_2$ be two $*$-rings with $V(\Gamma^{*}(R_1))\neq \phi$ and $V(\Gamma^{*}(R_2))\neq \phi$, and  $R=R_1 \oplus R_2$ is a $*$-ring with a component-wise involution. Then  $\Gamma^{*}(R)$ is connected and diam$(\Gamma^{*}(R))\leq 4$.
\end{theorem}
Since for a $*$-ring $R$, $\Gamma^*(R)$ is a subgraph of $\Gamma'(R)$ with the same vertex set, 
as a consequence of the above theorem we get the following corollary for $\Gamma'(R)$.
\begin{corollary}
Let $R_1,~ R_2$ be two $*$-rings with $V(\Gamma^{'}(R_1))\neq \phi$ and $V(\Gamma^{'}(R_2))\neq \phi$, and  $R=R_1 \oplus R_2$ is a $*$-ring with a component-wise involution. Then  $\Gamma^{'}(R)$ is connected and diam$(\Gamma^{'}(R))\leq 4$.
\end{corollary}

\begin{corollary}
Let $R$ be a $*$-ring. If $R$ contains a non-trivial central projection, then $\Gamma^{'}(R)$ is connected, and diam$(\Gamma^{'}(R))\leq 4$.
\end{corollary}
\begin{proof}
Suppose $R$ contains a non-trivial central projection say $e$. Then $R=R_1\oplus R_2$ where $R_1=eR$ and $R_2=(1-e)R$ are $*$-rings. The result follows from the Theorem \ref{thm5.9}.
\end{proof}

\end{document}